\documentclass[12pt]{amsart}

\usepackage{amsmath}
\usepackage{amssymb}
\usepackage{amsopn}
\usepackage{epsfig}
\usepackage{amsfonts}
\usepackage{latexsym}
\usepackage{enumitem}
\usepackage{graphicx}

\newtheorem{theorem}{Theorem}[section]
\newtheorem{lemma}[theorem]{Lemma}
\newtheorem{proposition}[theorem]{Proposition}

\theoremstyle{definition}

\theoremstyle{remark}
\newtheorem{remark}[theorem]{Remark}

\numberwithin{equation}{section}

\newcommand{\B}{\ensuremath{\mathcal{B}}}

\newcommand{\U}{\mathcal{U}}

\newcommand{\set}[1]{\left\{#1\right\}}

\newcommand{\ga}{\gamma}

\newcommand{\f}{\infty}

\newcommand{\al}{\alpha}

\newcommand{\ra}{\rightarrow}

\begin{document}

\title{Smallest bases of expansions with multiple digits}

\author{Derong Kong}
\address{School of Mathematical Science, Yangzhou University, Yangzhou, Jiangsu 225002, People's Republic of China.}
\email{derongkong@126.com}
\author{Wenxia Li}
\address{Department of Mathematics, Shanghai Key Laboratory of PMMP, East China Normal University, Shanghai 200062,
People's Republic of China}
\email{wxli@math.ecnu.edu.cn}

\author{Yuru Zou}
\address{College of Mathematics and Computational Science, Shenzhen University, Shenzhen 518060, People's Republic of China}
\email{yuruzou@szu.edu.cn}

 \renewcommand{\thefootnote}{}
\footnotetext{Correspond author: Yuru Zou, email:yuruzou@szu.edu.cn}
\date{\today}
\dedicatory{}



\begin{abstract}
Given two   positive integers   $M$ and  $k$, let $\B_k$ be the set of bases $q>1$ such that there exists a real number $x$ having exactly $k$ different $q$-expansions over the alphabet $\{0,1,\cdots,M\}$. In this paper we investigate the smallest base $q_2$ of $\B_2$, and
  show that if $M=2m$ the smallest base
$$q_2 =\frac{m+1+\sqrt{m^2+2m+5}}{2},$$
and if $M=2m-1$ the smallest base $q_2$ is the appropriate root of
$$
    x^4=(m-1)\,x^3+2 m\, x^2+m \,x+1.
$$
Moreover,   for $M=2$ we
show that $q_2$ is also the smallest base  of  $\B_k$   for all
$k\ge 3$. This turns out to be   different from that for $M=1$.
\end{abstract}
 \maketitle

\section{Introduction}\label{sec: Introduction}
Given $M\ge 1$ and $q>1$, the infinite sequence $(d_i)=d_1d_2\cdots$ with $d_i\in\{0,1,\cdots,M\}, i\ge 1$, is called a $q$-expansion of $x$ if
$$
x=\sum_{i=1}^\f\frac{d_i}{q^i}.
$$
Here the alphabet  $\set{0,1,\cdots,M}$ will be fixed throughout  the paper.

Non-integer base  expansions were pioneered by R\'{e}nyi
\cite{Renyi_1957} and Parry \cite{Parry_1960}. It was first
discovered by Erd\H{o}s and Jo\'{o} \cite{Erdos_Joo_1992} that for
any $k\ge 2$ there exist $q\in(1,2)$  such that $1$ has exactly $k$
different $q$-expansion with alphabet $\{0,1\}$.

 For $M\ge 1$ and $k=1,2,\cdots,\aleph_0$ or $2^{\aleph_0}$ let $\B_k=\B_k(M)$ be the set of $q>1$ such that there exists $x\in I_q:=[0, M/(q-1)]$ having exactly $k$ different $q$-expansions.

When $M=1$, the following results were obtained  in    \cite{Sidorov_2009, Baker_Sidorov_2014, Zou_Kong_2015} .
\begin{theorem}\label{th:11}
Let $M=1$. Then
\begin{enumerate}[label={\rm(\alph*)}]
  \item the  smallest element of $\B_2(1)$ is
$$q_2(1)\approx 1.71064$$
the appropriate root of  $x^4=2x^2+x+1$.
  \item \[q_2(1)\in\B_1(1)\cap\B_2(1)\cap\B_{2^{\aleph_0}}(1)\quad\textrm{and}\quad q_2(1)\notin\B_{\aleph_0}(1)\cup\bigcup_{k=3}^\f\B_{k}(1).\]
\end{enumerate}
\end{theorem}

In this paper we  investigate the smallest base of $\B_2(M)$ for all $M\ge 1$.
\begin{theorem}\label{th:12}
  \begin{enumerate}[label={\rm(\alph*)}]
    \item If $M=2m$, then the smallest base  $q_2(M)$ of $\B_2(M)$ is given by
    $$q_2(M)=\frac{m+1+\sqrt{m^2+2m+5}}{2}.$$
    \item If $M=2m-1$, then the smallest base $q_2(M)$ of $\B_2(M)$ is  the appropriate root of
    $$
    x^4=(m-1)x^3+2 m \,x^2+m\, x+1.
    $$
  \end{enumerate}
\end{theorem}
In terms of Theorem \ref{th:12} we give  numerical calculations of $q_2=q_2(M)$ for $M=1,2,\cdots, 7$.
\begin{center}
 \begin{tabular}{|c|c|c|c|c|c|c|c|}
     \hline
     M & 1 & 2 & 3 & 4 & 5 & 6 & 7 \\\hline
     $q_2\approx$ & 1.71064& 2.41421& 2.75965& 3.30278& 3.80320& 4.23607 &4.83469 \\
     \hline
   \end{tabular}
   \end{center}

 By Theorem \ref{th:11} (b) it follows that for $M=1$ any $x\in I_{q_2(1)}$ can only be in one the following situations:
 (1) $x$ has a unique $q_2(1)$-expansion; (2) $x$ has exactly two different $q_2(1)$-expansions; (3) $x$ has a continuum of $q_2(1)$-expansions. One may expect that this   occurs for all $M\ge 1$.

 Our next result shows that this is not the case. In particular, for $M=2$ we show that for any $k=1,2,\cdots,\aleph_0$ or $2^{\aleph_0}$ there exists $x\in I_{q_{2}(2)}$ which has exactly $k$ different $q_2(2)$-expansions.
  \begin{theorem}\label{th:13}
 Let $M=2$. Then
 \[q_2(2)=1+\sqrt{2}\, \in\B_{2^{\aleph_0}}(2)\cap\B_{\aleph_0}(2)\cap\bigcap_{k=1}^\f\B_k(2).\]
Furthermore, $q_2(2)$ is the smallest element of $\B_k(2)$ for $k= 2, 3,\cdots$.
 \end{theorem}

The paper is arranged in the following way. In Section \ref{sec:unique expansion} we give  the detailed  descriptions of the unique expansions in small bases, which plays an important role in the investigation of  the smallest base $q_2$. The proof of Theorem \ref{th:12} will be presented in Section   \ref{sec:smallest base with M even} for $M=2m$  and in  Section \ref{sec:smallest base with M odd} for $M=2m-1$, respectively. In Section \ref{sec:multiple expansions} we prove Theorem \ref{th:13} and consider   some  questions on multiple expansions with multiple digits.

\section{unique expansions}\label{sec:unique expansion}
Let $M\ge 1$ and $q\in(1,M+1]$. It is well-known that almost every real number $x\in I_q$
has a continuum of $q$-expansions (cf.~\cite{Sidorov_2003, Dajani_DeVries_2007}). Surprisingly,
 in 1991 Erd\H{o}s   et al. \cite{Erdos_Horvath_Joo_1991} discovered that there exists infinitely many reals having a unique $q$-expansion.
  For more information on unique expansions we refer to \cite{Komornik_Loreti_2007, Allouche_Frougey_Hare_2007, DeVries_Komornik_2008, DeVries_Komornik_2011,Komornik_Kong_Li_2015_1, Kong_Li_2015}, and surveys \cite{Komornik_2011, Sidorov_2003}.

 For $q\in(1,M+1]$ let
 $\al(q)=(\al_i(q))$
 be the \emph{quasi-greedy} $q$-expansion of $1$ (cf.~\cite{Daroczy_Katai_1993}), i.e., the lexicographically largest infinite $q$-expansion of $1$. Here an expansion $(d_i)$ is called \emph{infinite} if $d_i\ne 0$ for infinitely many   $i\ge 1$.
 In this paper we will use lexicographical order between sequences and blocks.

For $q>1$ let $\U_q$ be the set of $x\in I_q$ having a unique $q$-expansion, and let $\U'_q$ be the set of corresponding expansions.
The following lexicographical characterization of  $\U_q'$ was essentially due to Parry \cite{Parry_1960} (see also, \cite{Baiocchi_Komornik_2007}).
\begin{lemma}\label{prop:21}
Let $q\in(1,M+1]$. Then an expansion $(d_i)\in\U_q'$ if and only if
\begin{equation*}
\left\{\begin{array}{lll}
  d_{n+1}d_{n+2}\cdots<\al(q) & \textrm{whenever} & d_n<M, \\
   d_{n+1}d_{n+2}\cdots>\overline{\al(q)} & \textrm{whenever} & d_n>0.
\end{array}\right.
\end{equation*}

 Moreover, the map $q\ra\al(q)$ is strictly increasing from $(1, M+1]$ onto the set of infinite sequences $(\ga_i)$ satisfying
 $$
 \ga_{i+1}\ga_{i+2}\cdots\le\ga_1\ga_2\cdots\quad\textrm{for all}~~ i\ge 0.
 $$
\end{lemma}
Here for a sequence $(d_i)$ we denote by $\overline{(d_i)}=(M-d_i)$ the \emph{reflection} of $(d_i)$. Accordingly,  we write for $\overline{c_1\cdots c_n}=(M-c_1)\cdots (M-c_n)$ the
reflection of $c_1\cdots c_n$.

For $M\ge 1$ we recall from \cite{Baker_2014} that the generalized golden ratio $p_1=p_1(M)$ is given by
\begin{equation}\label{eq:21}
\al(p_1)=\left\{  \begin{array}{lll}
    m^\f & \textrm{if}& M=2m, \\
    (m\,(m-1))^\f & \textrm{if} & M=2m-1. \\
  \end{array}\right.
\end{equation}

The following  lemma for $q$-expansions with $q\in(1,p_1]$ was established by
\cite{Baker_2014, Ge_Tan_2015_4}.
\begin{lemma}
  \label{prop:22}
  If $q\in(1,p_1)$, then all numbers $x\in(0,M/(q-1))$ has a continuum of $q$-expansions. If $q=p_1$, any number $x\in(0,M/(q-1))$ has either a continuum of $q$-expansions or countably infinitely many  $q$-expansions.
\end{lemma}

In order to investigate the smallest base of $\B_2(M)$ we need the  detailed description of $\U_q'$ for small bases $q$. By Lemma \ref{prop:22} it suffices to consider $\U'_q$ for $q>p_1$.
For $M\ge 1$ let
\begin{equation} \label{eq:22}
\al(p_2)=\left\{  \begin{array}{lll}
      ((m+1)(m-1))^\f & \textrm{if}& M=2m, \\
     (mm(m-1)(m-1))^\f & \textrm{if} & M=2m-1. \\
  \end{array}\right.
\end{equation}

When $M$ is even, the following proposition for $\U_q'$  was implicitly shown in  \cite[Lemma 4.12]{Kong_Li_Dekking_2010}.
\begin{proposition}\label{prop:23}
 If $M=2m$, then for $p_1<q\le p_2$ we have
$$
\U_q'=\set{0^\f, \overline{0^\f}}\cup\bigcup_{k=0}^\f\bigcup_{u=0}^m \set{0^k u m^\f, \overline{0^k u m^\f}}.
$$
\end{proposition}
\begin{proof}
First we consider the $``\supseteq"$ part. Note that $q>p_1$. Then
by (\ref{eq:21})  and Lemma \ref{prop:21} it
follows that
$$
\al(q)>\al(p_1)=m^\f.$$
Therefore,  the $``\supseteq"$ part can be verified by using Lemma \ref{prop:21}.


Now we consider the $``\subseteq"$ part.   Take $(d_i)\in\U_q'$ with $q\in(p_1,p_2]$. By symmetry we may assume that $d_1<M$. Apart from the trivial case that $(d_i)=0^\f$ let $n\ge 1$ be the least integer such that $d_n> 0$.
   Now we split the proof into the following two cases: (I) $n=1$; (II) $n>1$.

  Case (I). $n=1$. Then $0<d_n<M$. Note by (\ref{eq:22}) and Lemma \ref{prop:21} that
  \begin{equation}\label{eq:23}
  \al(q)\le\al(p_2)=((m+1)(m-1))^\f.
  \end{equation}
  Then by Lemma \ref{prop:21} it follows that
  $$d_{n+1}\in\set{m-1,m,m+1}.$$
  We claim $d_{n+1}d_{n+2}\cdots=m^\f$ by the following
  observations.
  \begin{itemize}
    \item If $d_{n+1}=m-1$, then  by using $d_n>0$ and (\ref{eq:23}) in Lemma \ref{prop:21} it follows that
    $$d_{n+1}d_{n+2}\cdots>\overline{\al(q)}\ge\overline{\al(p_2)}=((m-1)(m+1))^\f,$$
    which implies $d_{n+2}\ge m+1$.

     On the other hand, by using $d_{n+1}=m-1<M$ and (\ref{eq:23}) in  Lemma \ref{prop:21} we obtain
    $$
    d_{n+2}d_{n+3}\cdots<\al(q)\le\al(p_2)=((m+1)(m-1))^\f.
    $$
    Therefore, $d_{n+2}=m+1$.
    \item  If $d_{n+1}=m+1$, then by using $d_{n}<M$ and (\ref{eq:23}) in Lemma \ref{prop:21} it follows that
    $$
    d_{n+1}d_{n+2}\cdots<\al(q)\le\al(p_2)=((m+1)(m-1))^\f,
    $$
    which implies $d_{n+2}\le m-1$.

    On the other hand, by using $d_{n+1}=m+1>0$ and (\ref{eq:23}) in Lemma \ref{prop:21} it follows that
    $$
    d_{n+2}d_{n+3}>\overline{\al(q)}\ge\overline{\al(p_2})=((m-1)(m+1))^\f.
    $$
    Therefore, $d_{n+2}=m-1$.
   \end{itemize}

   By iteration of the above two arguments we conclude that if $d_{n+1}\in\set{m-1,m+1}$, then
   $(d_i)$ will eventually end with $((m+1)(m-1))^\f$, leading to a contradiction with Lemma \ref{prop:21} and (\ref{eq:23}).

   Hence,
\begin{equation}\label{eq:24}
(d_i)=d_1 m^\f\quad\textrm{ with}\quad0<d_1<M.
\end{equation}

Case (II). $n>1$. Since $d_{n-1}=0$, we have by using (\ref{eq:23}) in Lemma \ref{prop:21} that $d_{n}\in\set{1,\cdots,m+1}$. If $d_n=m+1$, then by the same arguments as in Case I it follows that
$$
(d_i)=0^{n-1}((m+1)(m-1))^\f,
$$
leading to a contradiction with Lemma \ref{prop:21} and (\ref{eq:23}).

Then  $0<d_n\le m<M$. In a similar way as in Case I we conclude that
\begin{equation}\label{eq:25}
(d_i)=0^{n-1}d_n m^\f\quad\textrm{ with}\quad 0<d_n\le m.
\end{equation}

By (\ref{eq:24}) and (\ref{eq:25}) we establish the $``\subseteq"$ part.
\end{proof}
When $M$ is odd, the following characterization of the set $\U_q'$
  was implicitly given in \cite[Proposition
13]{Glendinning_Sidorov_2001}.
\begin{proposition}
  \label{prop:24}
If $M=2m-1$, then for $p_1<q\le p_2$ we have
\begin{equation*}
\begin{split}
\U_q'=&\{0^\f,\overline{0^\f}\}\cup\bigcup_{k=0}^\f\bigcup_{u=0}^{m-1}\set{0^k u (m(m-1))^\f, 0^k u ((m-1)m)^\f}\\
\cup&\bigcup_{k=0}^\f\bigcup_{u=0}^{m-1}\set{  \overline{0^k u (m(m-1))^\f}, \overline{0^k u ((m-1)m)^\f}}.
\end{split}
\end{equation*}
\end{proposition}
\begin{proof}
For $m=1$ the proposition was established by Glendinning and Sidorov \cite{Glendinning_Sidorov_2001}. In the following we will assume $m\ge 2$.

The $``\supseteq"$ part can be easily verified by using Lemma \ref{prop:21} and (\ref{eq:21}). Then it suffices to prove the $``\subseteq"$ part.

  Take $(d_i)\in\U_q'$ with $q\in(p_1,p_2]$. By symmetry we assume $d_1<M$. Excluding  the trivial case that $(d_i)=0^\f$, let $n\ge 1$ be the least integer such that $d_n>0$.
   We split the proof into the following two cases: (I) $n=1$; (II) $n>1$.

  Case (I). $n=1$. Then $0<d_n<M$. Note by (\ref{eq:22}) and Lemma \ref{prop:21} that
  \begin{equation}
    \label{eq:26}
    \al(q)\le\al(p_2)=(mm(m-1)(m-1))^\f.
  \end{equation}
  By Lemma \ref{prop:21} it follows that $d_{n+1}\in\set{m-1,m}$. We claim that $d_{n+1}d_{n+2}\cdots$ equals $(m(m-1))^\f$ or its reflection $((m-1)m)^\f$.
  \begin{itemize}
    \item If $d_{n+1}d_{n+2}=(m-1)(m-1)$, then by using $d_n>0$ and (\ref{eq:26}) in Lemma \ref{prop:21} it follows that
    \begin{equation*}
    d_{n+1}d_{n+2}\cdots>\overline{\al(q)}\ge\overline{\al(p_2)}=((m-1)(m-1)mm)^\f,
    \end{equation*}
    which implies $d_{n+3}d_{n+4}\ge mm$.

     On the other hand, by using $d_{n+2}=m-1<M$ and (\ref{eq:26}) in Lemma \ref{prop:21} we have
 \begin{equation*}
    d_{n+3}d_{n+4}\cdots<\al(q)\le \al(p_2)=(mm(m-1)(m-1))^\f.
\end{equation*}
    Therefore, $d_{n+3}d_{n+4}=mm$.

    \item  If $d_{n+1}d_{n+2}=mm$, then by using $d_n<M$ and (\ref{eq:26}) in Lemma \ref{prop:21} it follows that
    \begin{equation*}
      d_{n+1}d_{n+2}\cdots<\al(q)\le\al(p_2)=(mm(m-1)(m-1))^\f,
    \end{equation*}
    which implies $d_{n+3}d_{n+4}\le (m-1)(m-1)$.

     On the other hand, by using $d_{n+2}=m>0$ and (\ref{eq:26}) in Lemma \ref{prop:21} it gives that
   \begin{equation*}
    d_{n+3}d_{n+4}\cdots>\overline{\al(q)}\ge\overline{\al(p_2)}=((m-1)(m-1)mm)^\f.
 \end{equation*}
    Therefore, $d_{n+3}d_{n+4}=(m-1)(m-1)$.
  \end{itemize}

  Hence, by the above arguments it follows that if
   $d_{n+1}d_{n+2}=mm$ or its reflection  $(m-1)(m-1),$  then $(d_i)$ will eventually end with $(mm(m-1)(m-1))^\f$, leading to a contradiction with Lemma \ref{prop:21} and (\ref{eq:26}).

  Therefore,
\begin{equation}\label{eq:27}
  (d_i)=d_1(m(m-1))^\f\quad\textrm{or}\quad d_1((m-1)m)^\f,
  \end{equation}
  where $0<d_1<M$.

  Case (II). $n>1$. Then by using $d_{n-1}=0<M$ in Lemma \ref{prop:21} it follows that
  $$
  d_{n}, d_{n+1}\in\set{1,\cdots,m}.
  $$
  If $d_n=m$, then  $d_n>0$, and by using (\ref{eq:26}) in Lemma \ref{prop:21} it follows that $d_{n+1}\ge m-1$. By the same arguments as in Case I it follows that
\begin{equation}\label{eq:28}
  (d_{i})=0^{n-1}(m(m-1))^\f.
  \end{equation}

  If   $0<d_n<m\le M$, then by a similar way as in Case (I) we conclude that
\begin{equation}\label{eq:29}
  (d_i)=0^{n-1}d_n(m(m-1))^\f\quad\textrm{or}\quad 0^{n-1}d_n((m-1)m)^\f,
  \end{equation}
  where $0<d_n<m$.

  Therefore,  by (\ref{eq:27})--(\ref{eq:29}) we prove the $``\subseteq"$ part.
\end{proof}

At the end of this section we give a geometrical explanation of expansions in non-integer bases. For $k=0,1,\cdots , M$ and $q\in(1,M+1]$ let
\[f_k(x)=\frac{x+k}{q}.\]
 Then   the interval $I_q=[0, M/(q-1)]$ can be written as
\begin{equation}\label{eq:212}
 I_q=\bigcup _{k=0}^M f_k(I_q)=\bigcup_{k=0}^M \left [\frac{k}{q}, \frac{M}{q(q-1)}+\frac{k}{q}\right ].
 \end{equation}

 Therefore, a point $x\in I_q$ has a unique $q$-expansion $(x_i(q))$ if and only if all of its linear transformations  $\sum_{i=1}^\f x_{j+i}(q)/q^i$, $j\ge 0$, do not fall into the overlaps, i.e., for all $j\ge 0$ we have
 \[
\sum_{i=1}^\f \frac{x_{j+i}(q)}{q^i}\notin f_{k_1}(I_q)\cap f_{k_2}(I_q)\quad\textrm{for any}\quad k_1<k_2.
 \]
 On the other hand, if $\sum_{i=1}^\f x_{j+i}(q)/q^i\in f_{k_1}(I_q)\cap f_{k_2}(I_q)$ for some $j\ge 0$ and $k_1<k_2$. Then
 $x$ has at least two expansions: one begins with $x_1(q)\cdots x_j(q)k_1$ and the other starts at $x_1(q)\cdots x_j(q)k_2$.

 By (\ref{eq:212}) and the definition of $p_1$ in (\ref{eq:21}) one can easily verify the following lemma.
 \begin{lemma}\label{lem:26}
   Let $q>p_1$. Then
   $f_{k_1}(I_q)\cap f_{k_2}(I_q)\cap f_{k_3}(I_q)=\emptyset$
    for any $k_1<k_2<k_3$.
 \end{lemma}

\section{smallest base  of $\B_2(M)$ with $M=2m$}\label{sec:smallest base with M even}
In this section we will consider the smallest base $q_2(M)$ of $\B_2(M)$ with  $M=2m$.
For $q>1$ and an expansion $(d_i)$ we set
$$
((d_i))_q:=\sum_{i=1}^\f\frac{d_i}{q^i}.
$$
Recall from (\ref{eq:21}) and (\ref{eq:22}) that
\begin{equation}\label{eq:31}
p_1=m+1,\quad p_2=\frac{m+1+\sqrt{m^2+6m+1}}{2}.
\end{equation}

\begin{lemma}\label{lem:31}
Let $M=2m$ and $q\in(p_1,p_2]$. Then $q\in\B_2(M)$ if and only if $q$ is a root  of
$$
(10^k u m^\f)_q=(0\overline{0^j v m^\f})_q,
$$
for some $k,j= 0,1,\cdots$ and $u,v\in\{0,\cdots,m\}$.
\end{lemma}
\begin{proof}
First we point out that $\B_2(M)\cap(p_1,p_2]\ne\emptyset$. In fact, by  Proposition \ref{prop:23} and Lemma \ref{lem:26} it follows that
$$
x=(100m^\f)_{p_2}=(0 \overline{(m-1)m^\f})_{p_2}
$$
has exactly two $p_2$-expansions, i.e., $p_2\in\B_2(M)$.

For the sufficiency we assume that $(10^k u m^\f)_q=(0\overline{0^j v m^\f})_q$ for some $k,j= 0,1,\cdots$ and $u,v\in\{0,\cdots,m\}$. Then
\[x:= (10^k u m^\f)_q=(0\overline{0^j v m^\f})_q\]
has
at least two different $q$-expansions. Let $(x_i)$ be a $q$-expansion of $x$.
Then $x_1\in \set{0, 1}$ by Lemma \ref{lem:26}. When $x_1 = 1$,   by Proposition \ref{prop:23} it yields that $q x-1 =
(0^k u m^\f)_q$ has a unique $q$-expansion.   When $x_1 = 0$,   by Proposition \ref{prop:23} we also have that $qx = (\overline{0^jv m^\f})_q$
has a unique $q$-expansion. Thus $x$ has
exactly two different $q$-expansions and so $q\in\B_2(M)$.

Now we consider  the necessity.
Take $q \in(p_1,p_2]\cap\B_2(M)$.
Then there exists $x\in I_q$ which has exactly two different $q$-expansions $(a_i)$
and $(b_i)$, i.e.,
\begin{equation}\label{e31}
((a_i))_{q}=x=((b_i))_{q}.
\end{equation}

Let $n\ge 1$ be the least integer such that $a_n\ne b_n$. Without loss of generality we assume  $a_n> b_n$. Then
by (\ref{e31}) it follows
$$
(a_n a_{n+1}\cdots )_q=(b_nb_{n+1}\cdots )_q
\quad\textrm{and}\quad(a_{n+i}), ~(b_{n+i})\in\U_q'.
$$
By Lemma \ref{lem:26} we have $a_n=b_n+1$, and therefore
$$
\frac{1}{q}=\frac{1}{q}\sum _{k=1}^\infty \frac{b_{n+k}-a_{n+k}}{q^k}\leq \frac{b_{n+1}-a_{n+1}}{q^2}+\sum _{k=3}^\infty \frac{2m}{q^k}.
$$
This, together with  $q>p_1=m+1$, implies that $a_{n+1}<b_{n+1}$.
Hence,
\begin{equation}\label{eq:32}
(1a_{n+1}a_{n+2}\cdots )_q=(0b_{n+1}b_{n+2}\cdots )_q,
\end{equation}
where
$a_{n+1}<b_{n+1} $ and $(a_{n+i}), (b_{n+i})\in\U_q'.$

Now we claim  that $(a_{n+i})$ and $(b_{n+i})$ are of the forms $0^k u
m^\f$ and  $\overline{0^j v m^\f}$, respectively, where $k, j=0,1,\cdots$ and $u,v \in\set{0,1,\cdots,m}$. In terms of Proposition \ref{prop:23} it suffices to prove that    $(a_{n+i})$ and $(b_{n+i})$ can not be both of the form $0^k u
m^\f$ or  $\overline{0^j v m^\f}$.

\begin{itemize}
  \item If both $(a_{n+i})$ and $(b_{n+i})$ are of form $0^k u
m^\f$, then by (\ref{eq:32}) it gives that
$$
 (1a_{n+1}a_{n+2}\cdots )_q=(0b_{n+1}b_{n+2}\cdots )_q
$$
has at least two $q$-expansions. This leads to a contradiction, since by Proposition \ref{prop:23} that the number $(0b_{n+1}b_{n+2}\cdots )_q$ should have a unique $q$-expansion.
  \item If both $(a_{n+i})$ and $(b_{n+i})$ are of form $\overline{0^j v m^\f}$, then by (\ref{eq:32}) it follows that
$$
  (M a_{n+1}a_{n+2}\cdots )_q=((M-1)b_{n+1}b_{n+2}\cdots )_q
$$
has at least two $q$-expansions. This also leads to a contradiction, since by Proposition \ref{prop:23} that the number $ (M a_{n+1}a_{n+2}\cdots )_q$ should have a unique $q$-expansion.
\end{itemize}

Finally,     we show that $k,j\ne \infty $. Without loss of generality we only prove $k\ne \f$.

Suppose on the contrary that $k=\infty $.
 Then
 $
 1=((2m)^j(2m-v)m^\infty )_q
 $ for some $j=0,1,\cdots,\f$ and $v\in\set{0,1,\cdots,m}$.
 By Lemma \ref{prop:21} it follows that
 \[
 \alpha (q)=(2m)^j(2m-v)m^\infty.
 \]
 This leads to a contradiction with $p_1<q\le p_2$ and Lemma \ref{prop:21} that
 $$
m^\f=\al(p_1)<\al(q)\le \al(p_2)= ((m+1)(m-1))^\infty.
 $$
\end{proof}

In terms of Lemma \ref{lem:31} to find the smallest base $q_2$ of $\B_2(M)$ it suffices to consider the zeros in $(p_1,p_2]$ of the function
\begin{equation}\label{e33}
\begin{split}
 f_{k,j,u,v}(q)&=(q^3-q^2)\big((10^k u
m^\f)_q-(0\overline{0^j v m^\f})_q\big)\\
 & =-q-2mq+q^2+q^{-k}(m-u+u q)\\
 &\hspace{4cm}+q^{-j}(m-v+v q  ).
\end{split}
\end{equation}

\begin{lemma}\label{lem:32}
Given $k,j\ge 0$ and $u,v\in\{0,1,\cdots,m\}$, the function
 $f_{k,j,u,v}(q)$
is  strictly increasing in $(p_1,\f)$.
\end{lemma}
\begin{proof}
Differentiating $f_{k,j,u,v}$ in (\ref{e33})  it gives
\begin{equation*}
\begin{split}
 f_{k,j,u,v}'(q)&= -1-2m+2q+q^{-k}\left(  u-ku+
 \frac{ku-km}{q}\right) \\
 &\hspace{3.5cm}+q^{-j}\left( v-jv+
 \frac{jv-jm}{q}\right).
\end{split}
\end{equation*}
Since $q> p_1=m+1$, we have $-1-2m+2q> 1$. In order to guarantee the positivity of $f'_{k,j,u,v}(q)$, by symmetry it suffices to prove
\begin{equation}\label{e34}
q^{-k}\left( u-ku+
 \frac{ku-km}{q}\right)\ge -\frac{1}{2}
 \end{equation}
for any $k\ge 0$ and $u\in\set{ 0,\cdots,m}$.

Clearly,  the inequality (\ref{e34}) holds for $k=0$ or $1$. For $k\geq 2$ we have by using $q>p_1=m+1$ that
 $$1-k+\frac{k}{q}\le1-k+\frac{k}{m+1}\le 1-\frac{k}{2}\le 0,$$
  and therefore
\begin{equation*}
\begin{split}
q^{-k}\left( u-ku+
 \frac{ku-km}{q}\right)
&= -\frac{km}{q^{k+1}} +\frac{u}{q^{k}}\left( 1-k+\frac{k}{q} \right)\\
&\geq -\frac{km}{q^{k+1}}+\frac{m}{q^{k}}\left( 1-k+\frac{k}{q} \right)\\
&=\frac{m(1-k)}{q^k}>\frac{1-k}{q^{k-1}}\ge -\frac{1}{2}.
\end{split}
\end{equation*}
This establishes (\ref{e34}).
\end{proof}

\begin{lemma}
  \label{lem:33}
  The equation $f_{k,j,u,v}(q)=0$   has a unique root in $(p_1,\f)$ if and only if the parameters $(k,j,u,v)$ satisfies
  \begin{equation}\label{e35}
   \frac{u+1}{(m+1)^{k+1}}+\frac{v+1}{(m+1)^{j+1}}<1.
  \end{equation}
  \end{lemma}
\begin{proof}
  By Lemma \ref{lem:32} and the continuity of $f_{k,j,u,v}$ it follows that the equation $f_{k,j,u,v}(q)=0$   has a unique root in $(p_1,\f)$ if and only if
   $$f_{k,j,u,v}(p_1)<0.$$
   Hence, the lemma follows by using (\ref{eq:31}) in (\ref{e33}) that
  \begin{equation*}
    \begin{split}
      f_{k,j,u,v}(p_1) = {m}{(m+1)}\left(\frac{u+1}{(m+1)^{k+1}}+\frac{v+1}{(m+1)^{j+1}} -1\right).
    \end{split}
  \end{equation*}
\end{proof}

In terms of Lemma \ref{lem:33}    we denote by $q_{k,j,u,v}$ the unique root of $f_{k,j,u,v}(q)=0$ in $(p_1,\f)$, where $(k, j,u,v)$ satisfies (\ref{e35}).

\begin{lemma}\label{lem:34}
\begin{enumerate}
  \item  The sequence $(q_{k,j,u,v})$ is strictly increasing with respect to  the  parameters $k$ and $j$;
  \item The sequence $(q_{k,j,u,v})$ is strictly decreasing with respect to  the  parameters $u$ and $v$.
\end{enumerate}
    \end{lemma}
\begin{proof}
Without loss of generality  we only give the proof for the parameter $k$.

 Fix $j\ge 0$ and $u, v\in\set{0,1,\cdots,m}$. We   write $q_k=q_{k,j,u,v}$. Then by (\ref{e33}) we have
 \begin{equation*}
  \begin{split}
 f_{k+1,j,u,v}(q_{k+1})&=-q_{k+1}-2mq_{k+1}+q_{k+1}^2+q_{k+1}^{-k-1}(m-u+q_{k+1}u)\\
 &\quad\quad+q_{k+1}^{-j}(m-v+q_{k+1}v)\\
&<-q_{k+1}-2mq_{k+1}+q_{k+1}^2+q_{k+1}^{-k}(m-u+q_{k+1}u)\\
&\quad\quad+q_{k+1}^{-j}(m-v+q_{k+1}v)\\
&=f_{k,j,u,v}(q_{k+1}),
  \end{split}
  \end{equation*}
  where the strictly   inequality holds since $m-u+q_{k+1}u> 0$.
This, together with $f_{k+1,j,u,v}(q_{k+1})=0=f_{k,j,u,v}(q_k)$, implies that
  $$
  f_{k,j,u,v}(q_k)<f_{k,j,u,v}(q_{k+1}).
  $$
Therefore, by Lemma \ref{lem:32} it follows that $q_{k}<q_{k+1}$.
\end{proof}
Now we consider the bases  in $\B_2(M)\cap(p_1,p_2]$ with $M=2m$.

\begin{theorem}\label{th:35}
Let $M=2m$. Then
$$
\B_2(M)\cap(p_1,p_2]=\set{q_{1,0,u,m-1}: u=0,1,\cdots,m-1}.
$$

Furthermore, the smallest base  of $\B_2(M)$
is
$$q_2(M)=q_{1,0,m-1,m-1}=\frac{m+1+\sqrt{ m^2+2m+5}}{2}.$$
\end{theorem}
\begin{proof}
By Lemma \ref{lem:31} it suffices to investigate the  parameters $(k,j,u,v)$   such that
 $$p_1<q_{k,j,u,v}\le p_2.$$
Note by (\ref{eq:31}) that $p_2^2=(m+1)p_2+m$. Then by Lemmas \ref{lem:32}, \ref{lem:33} and by  (\ref{e33}) it follows that
 $q_{k,j,u,v}\in\B_2(M)\cap(p_1,p_2]$ if and only if the parameters $(k,j,u,v)$ satisfy
 \[f_{k,j,u,v}(p_1)<0,\quad f_{k,j,u,v}(p_2)\ge 0,\]
 i.e.,
 the parameters  $(k,j,u,v)$ satisfy
 (\ref{e35}) and
 \begin{equation}
   \label{e36}
   m(1-p_2)+\frac{m-u+u p_2}{p_2^k}+\frac{m-v+v p_2}{p_2^j}\ge 0.
 \end{equation}
Note by (\ref{e33})   that $f_{k,j,u,v}(q)=f_{j,k,v,u}(q)$. Then we may assume $k\ge j$.

If $m=1$, then   by (\ref{e33}) and Lemma \ref{lem:34} one can verify   that $q_{k,j,u,v}\in(p_1,p_2]$ if and only if
\[(k,j,u,v)\in\set{(2,1,1,1),(2,0,1,0),(1,1,1,0),(1,1,0,1),(1,0,0,0)}.\]
 Note that $q_{2,1,1,1}=q_{2,0,1,0}=q_{1,1,1,0}=q_{1,1,0,1}=q_{1,0,0,0}=1+\sqrt{2}$. Therefore,
  \[
\B_2\cap(p_1, p_2]=\set{q_{1,0,0,0}}=\set{1+\sqrt{2}}.
  \]

In the following we will assume  $m\ge 2$. First we show that $j=0$. Note by
 (\ref{e33})   that
 $$
 q_{1,1,m,m}  =2m>p_2.
 $$
Then by Lemma \ref{lem:34} we have $j=0$ as required. And so we have
$v \le m-1$ by (\ref{e35}). Moreover, one can check that $q_{2,0,m,m-1} = p_2$. By Lemma \ref{lem:34} this implies that $q_{2,0,u,m-1} > p_2$ for $u< m$, and that $q_{k,0,u,m-1}> p_2$
for $k\ge  3$. Note that $q_{1,0,0,m-1}=p_2$.  Hence, it suffices to consider
$k \le 1$.

If $k=j=0$, then by (\ref{e33}) we have
$$q_{0,0,u,v}=2m-u-v,$$  which can not fall into the interval $(p_1,p_2]$,  since by (\ref{eq:31}) we have
  $$
 m+1=p_1< p_2=\frac{m+1+\sqrt{(m+1)^2+4m}}{2}<m+2.
  $$

If  $k=1,j=0$, then by (\ref{e33}) we have
  $$
  q_{1,0,u,v}=\frac{2m-v+\sqrt{(2m-v)^2+4(m-u)}}{2}.
  $$
  One can check that $q_{1,0,u,v}\in(p_1,p_2]$ if and only if $v=m-1$ and $0\le u\le m-1$.
\end{proof}

\section{smallest base of $\B_2(M)$ with $M=2m-1$}\label{sec:smallest base with M odd}
In this section we are going to investigate  the smallest base $q_2(M)$ of $\B_2(M)$ with $M=2m-1$.
Recall from (\ref{eq:21}) and (\ref{eq:22}) that
\begin{equation}\label{e41}
p_1=\frac{m+\sqrt{m^2+4m}}{2},
\end{equation}
and $p_2\in(p_1,\f)$ satisfies
\begin{equation}\label{e42}
p_2^3=(m+1)p_2^2-p_2+m.
\end{equation}

\begin{lemma}\label{l41}
Let $M=2m-1$ and $q\in(p_1,p_2]$. Then $q\in\B_2(M)$ if and only if $q$ satisfies one of the following equations:
\begin{eqnarray}
(10^{k_1} u_1 (m(m-1))^\f)_q&=(0\overline{0^{j_1} v_1 (m(m-1))^\f})_q\label{eq:41}\\
(10^{k_2} u_2 ((m-1)m)^\f)_q&=(0\overline{0^{j_2} v_2 ((m-1)m)^\f})_q\label{eq:42}\\
(10^{k_3} u_3 (m(m-1))^\f)_q&=(0\overline{0^{j_3} v_3 ((m-1)m)^\f})_q\label{eq:43}\\
(10^{k_4} u_4 ((m-1)m)^\f)_q&=(0\overline{0^{j_4} v_4 (m(m-1))^\f})_q\label{eq:44}
\end{eqnarray}
for some parameters $k_i,j_i= 0,1,\cdots$, and  $u_i, v_i\in\set{0,1,\cdots,m-1}$, where $i=1,2,3,4,$.
\end{lemma}
\begin{proof}
 First we show that $\B_2(M)\cap(p_1,p_2]\ne\emptyset$.   By Proposition \ref{prop:24} and   Lemma \ref{lem:26} it follows that the number
  $$
  (10^4((m-1)m)^\f)_{p_2}=(0\overline{(m-1)((m-1)m)^\f})_{p_2}
  $$
  has exactly two $p_2$-expansions, i.e., $p_2\in \B_2(M)$.

Clearly, the sufficiency follows by Proposition \ref{prop:24} and Lemma \ref{lem:26}.

 Now we prove the necessity. Take $q \in(p_1,p_2]\cap\B_2(M)$.
Then there exists $x\in I_q$ which has exactly two different $q$-expansions $(a_i)$
and $(b_i)$, i.e.,
\begin{equation}\label{li5}
((a_i))_{q}=x=((b_i))_{q}.
\end{equation}
Let $n\ge 1$ be the least integer such that $a_n\ne b_n$. Without loss of generality we assume  $a_n> b_n$. Then
by (\ref{li5}) it follows
$$
(a_na_{n+1}\cdots )_q=(b_nb_{n+1}\cdots )_q \quad\textrm{and}\quad(a_{n+i}), ~(b_{n+i})\in\U_q'
$$
By  Lemma \ref{lem:26} we have $a_n=b_n+1$, and therefore
 \begin{equation}\label{li6}
1=(b_{n+1}b_{n+2}\cdots )_q-(a_{n+1}a_{n+2}\cdots )_q
\end{equation}
where $(a_{n+i}), ~(b_{n+i})\in\U_q'$.

Now we claim that $(a_{n+i})$ is of the form $0^{k} u (m(m-1))^\f$ or $0^{k} u ((m-1)m)^\f$,  and $(b_{n+i})$ is of the form $\overline{0^{k} u (m(m-1))^\f}$ or $\overline{0^{k} u ((m-1)m)^\f}$,
where $k=0,1,2,\cdots $ and $u\in \{0,1,\cdots , m-1\}$. This can be verified by the following observations.

 \begin{itemize}
   \item If $(b_{n+i})$ is  of the form $0^{k} u (m(m-1))^\f$  or $0^{k} u ((m-1)m)^\f$ with $k=0,1,\cdots, \infty$ and   $u\in \{0,1,\cdots , m-1\}$, then  by (\ref{li6}) and (\ref{eq:21}) it follows that
\begin{equation*}
\begin{split}
1&=(b_{n+1}b_{n+2}\cdots )_q-(a_{n+1}a_{n+2}\cdots )_q\leq ((m(m-1))^\infty)_{q} \\
&<((m(m-1))^\infty)_{p_1}=1,
\end{split}
\end{equation*}
leading to a contradiction.
   \item If both $(a_{n+i})$ and $(b_{n+i})$ are   of the form $ \overline{0^{k} u (m(m-1))^\f}$ or $\overline{0^{k} u ((m-1)m)^\f}$ with $k=0,1,\cdots, \infty$ and $u\in \{0,1,\cdots , m-1\}$, then by (\ref{li6}) and (\ref{eq:21}) it follows
\begin{equation*}
\begin{split}
1&=(b_{n+1}b_{n+2}\cdots )_q-(a_{n+1}a_{n+2}\cdots )_q \\
&\leq ((2m-1)^\infty )_q-((m(m-1))^\infty)_{q} <((m(m-1))^\infty)_{p_1}=1,
\end{split}
\end{equation*}
also leading to a contradiction.
 \end{itemize}
By the above arguments it follows that $(a_{n+i})$ is of the form  $0^{k} u (m(m-1))^\f$ or $0^{k} u ((m-1)m)^\f$, and $(b_{n+i})$ is of the form $\overline{0^{k} u (m(m-1))^\f}$ or $\overline{0^{k} u ((m-1)m)^\f}$, where $k=0,1,\cdots,\f$ and $u\in \{0,1,\cdots , m-1\}.$ Hence, to finish the proof it suffices to prove $k\ne \f$.

Case I. $(a_{n+i})=0^\infty $.  Note that $(b_{n+i})$ is of the form $\overline{0^{k} u (m(m-1))^\f}$ or $\overline{0^{k} u ((m-1)m)^\f}$ with $k=0,1,\cdots,\f$ and $u\in\set{0,1,\cdots,m-1}$. Then by (\ref{li6}) it follows that
\begin{equation*}
\begin{split}
\al(q)=b_{n+1}b_{n+2}\cdots.
\end{split}
\end{equation*}
This leads to a contradiction with using $p_1<q\le p_2$ in  Lemma \ref{prop:21} that
\[(m(m-1))^\f=\al(p_1)<\al(q)\le\al(p_2)=(mm(m-1)(m-1))^\f.\]

Case II.  $(b_{n+i})=\overline{0^\infty}=(2m-1)^\infty $. Note that $(a_{n+i})$ is of the form $0^{k} u (m(m-1))^\f$ or $0^{k} u ((m-1)m)^\f$ with $k=0,1,\cdots,\f$  and $u\in\set{0,1,\cdots,m-1}$. Then by (\ref{li6}) it follows that
\begin{equation*}
\begin{split}
\al(q)=\overline{a_{n+1}a_{n+2}\cdots}.
\end{split}
\end{equation*}
By a similar argument as in Case I one can   show that $q\notin(p_1, p_2]$, leading to  a contradiction.
\end{proof}

\begin{remark}\label{re:1}
We point out that (\ref{eq:43}) and $(\ref{eq:44})$ are equivalent.
In fact, if $q$ is a root of (\ref{eq:43}) for some $k_3,j_3\ge 0$
and $u_3,v_3\in\set{0,1,\cdots,m-1}$. Then by reflection we have
$$
(10^{j_3}v_3 ( (m-1)m)^\f)_q=(0\overline{0^{k_3}u_3(m(m-1))^\f})_q.
$$
This corresponds to (\ref{eq:44}) with $(k_4,j_4,u_4,v_4)=(j_3,k_3,v_3,u_3)$.
\end{remark}
By Lemma \ref{l41} and Remark \ref{re:1} it suffices to investigate the appropriate roots of Equations (\ref{eq:41})--(\ref{eq:43}), which will be dealt with separately  in the following three subsections.

\subsection{Solutions of  Equation (\ref{eq:41}).}\label{subs:1}

Given $k, j\ge 0$ and $0\le u,v\le  m-1$, by Lemma \ref{l41} it is necessary to consider the zeros in $(p_1, p_2]$ of the function
  \begin{equation}\label{e45}
  \begin{split}
g^{(1)}_{k,j,u,v}(q)&=(q^3-q)\big((10^k u (m(m-1))^\f)_q\\
&\;\;\;-(0\overline{0^j v (m(m-1))^\f})_q\big)\\
&=(q+1)(q-2m)+q^{-k-1}(m-1-u+mq+u q^2)\\
  &\hspace{3cm}+q^{-j-1}(m-1-v+mq+vq^2).
  \end{split}
  \end{equation}

\begin{lemma}
  \label{l43}
For any $k, j\ge 0$ and $u,v\in\set{0,1,\cdots,m-1}$ the  function
  $
  g^{(1)}_{k,j,u,v}(q)$
  is strictly increasing in $(p_1,\f)$.
\end{lemma}
\begin{proof}
In terms of (\ref{e45}) and by symmetry it suffices to prove
$$
h_{k,u}(q)=\frac{(q+1)(q-2m)}{2}+q^{-k-1}(m-1-u+mq+u q^2)
$$
has a positive derivative in $(m+1/2,\f)$ for any $k\ge 0$ and $u\in\set{0,1,\cdots,m-1}.$
Differentiating $h_{k,u}$ it yields
\begin{equation}\label{e46}
\begin{split}
h'_{k,u}(q)&=q-m+\frac{ 1}{2}+u q^{-k}\left( \frac{k+1}{q^2}+1-k  \right)\\
& \quad+q^{-k}\left(-(k+1)\frac{m-1}{q^2}-\frac{km}{q}\right).
\end{split}
\end{equation}
Then by using $q>p_1$ and $p_1^2=m p_1+m$ in (\ref{e46}) one can show that
$    h_{k,u}'(q)>0 $
for $k=0,1$ and $2$.

If $k\ge 3$, then by using $q>p_1$ we have
$
({k+1})/{q^2}+1-k\le 0.
$ Moreover, one can show that the function
\[
\phi(k)=q^{-k}\left( \frac{km}{q}+(m-1)(k-1) \right)
\] satisfies  $\phi(k+1)<\phi(k)$ for any $k\geq 2$.
Therefore, by using  $q>p_1$ and $p_1^2=m p_1+m$ in (\ref{e46}) it follows that
\begin{equation*}
  \begin{split}
   h'_{k,u}(q)&\ge  q-m+\frac{1}{2}-q^{-k}\left( \frac{km}{q}+(m-1)(k-1) \right)\\
     &\ge q-m+\frac{1}{2}-q^{-2}\left( \frac{2m}{q}+(m-1)  \right) \\
    &\ge p_1-m+\frac{1}{2}-\frac{m-1}{p_1^2}-\frac{2m}{p_1^3}\\
    &=p_1^{-3}\left(\frac{1}{2}p_1^3+p_1-m\right)\;>0.
  \end{split}
\end{equation*}
\end{proof}
%
%
%

\begin{lemma}
  \label{l44}
The equation $g^{(1)}_{k,j,u,v}(q)=0$   has a unique root in $(p_1,\f)$ if and only if the parameters $(k,j,u,v)$ satisfy
\begin{equation}
  \label{e47}
\frac{u+1}{m p_1^k}+\frac{v+1}{m p_1^j}<1.
\end{equation}
\end{lemma}
\begin{proof}
By the monotonicity shown in Lemma 4.3
 and the continuity of $g^{(1)}_{k,j,u,v}$ it follows  that the equation $g^{(1)}_{k,j,u,v}(q)=0$   has a unique root in $(p_1,\f)$ if and only if    $$g^{(1)}_{k,j,u,v}(p_1)<0.$$
Then the lemma follows by using $p_1^2=m p_1+m$ in (\ref{e45}) that
$$
 g^{(1)}_{k,j,u,v}(p_1)=\frac{p_1-1}{p_1}\left( -1+\frac{u+1}{m p_1^k}+\frac{v+1}{m p_1^j}\right).
$$
\end{proof}

In terms of Lemmas \ref{l44}      we denote by $q^{(1)}_{k,j,u,v}$ the unique root of $g^{(1)}_{k,j,u,v}(q)=0$ in $(p_1,\f)$, where $(k,j,u,v)$ satisfies (\ref{e47}).

In a similar way as in Lemma \ref{lem:34} one can verify the
following monotonicity of the sequence $(q^{(1)}_{k,j,u,v})$.
\begin{lemma}
  \label{l46}
  \begin{enumerate}
    \item  The sequence $(q^{(1)}_{k,j,u,v})$ is strictly increasing with respect to the  parameters $k$ and $j$;
    \item  The sequence $(q^{(1)}_{k,j,u,v})$ is strictly decreasing with respect to the parameters $u$ and $v$.
  \end{enumerate}
\end{lemma}

In the following lemma we show that no elements of $(p_1,p_2]\cap\B_2(M)$ satisfy Equation (\ref{eq:41}).

\begin{lemma}
  \label{l47}
   Let $M=2m-1$. Then equation (\ref{eq:41}) has no solutions in $(p_1,p_2]$.
\end{lemma}
\begin{proof}
By Lemmas \ref{l43} and \ref{l44} it suffices to prove that no
parameters $(k,j,u,v)$ satisfy both (\ref{e47}) and
$g^{(1)}_{k,j,u,v}(p_2)\ge 0$. Note by (\ref{e45}) that
$g^{(1)}_{k,j,u,v}(q)=g^{(1)}_{j,k,v,u}(q)$. Then we may assume that
$k\ge j$. Therefore, the lemma follows by observing  the following three
cases.

Case I.   $k\ge j\ge 1$. Then by Lemma \ref{l46} it suffices to prove
\[q^{(1)}_{1,1,m-1,m-1}>p_2,\] or equivalently, $g^{(1)}_{1,1,m-1,m-1}(p_2)<0$. This can be verified by using (\ref{e42}) in (\ref{e45}) that
\begin{equation*}
  \begin{split}
    g^{(1)}_{1,1,m-1,m-1}(p_2)&={p_2}^{-1}\big({p_2}^3-(2m-1){p_2}^2-2{p_2} +2m  \big)\\
    &\le {p_2}^{-1}\big((2-m){p_2}^2-2({p_2}-m)\big)\;< 0.
  \end{split}
\end{equation*}

Case II. $k>j=0$. Then by Lemma \ref{l46} it suffices to prove
$ q^{(1)}_{k,0,m-1,m-1}\le p_1$ for all $k\ge 1$, {and} $ q^{(1)}_{1,0,m-2,m-1}=q^{(1)}_{1,0,m-1,m-2}>p_2.$

By (\ref{e45}) and (\ref{e41}) one can show that
$$g_{k,0,m-1,m-1}^{(1)}(p_1)=\frac{m}{p_1^k}+\frac{m-1}{p_1^{k-1}}> 0.$$
By Lemma \ref{l43} this implies $q^{(1)}_{k,0,m-1,m-1}< p_1$.

Moreover,  by using (\ref{e42}) in (\ref{e45}) it follows that
\begin{equation*}
  \begin{split}
    g^{(1)}_{1,0,m-1,m-2}(p_2)=g^{(1)}_{1,0,m-2,m-1}(p_2)=-\frac{(1 + m - p_2) (p_2^2-1)}{p_2}<0.
  \end{split}
\end{equation*}
  Therefore, $q^{(1)}_{1,0,m-2,m-1}>p_2.$

Case III. $k=j=0$. Then by (\ref{e45}) it follows that
\begin{equation*}
  \begin{split}
 q^{(1)}_{0,0,u,v}&=\frac{2m-u-v-2}{2}+\frac{\sqrt{(2m-u-v)^2-4}}{2}.
  \end{split}
\end{equation*}
By using (\ref{e41}) and (\ref{e42}) one can show that
 $q^{(1)}_{0,0,u,v}\notin(p_1,p_2]$ for any $u, v\in\set{0,1,\cdots,m-1}$.
\end{proof}

 \subsection{Solutions of  Equation (\ref{eq:42})}

Given $k, j\ge 0$ and $0\le u,v\le m-1$, by Lemma \ref{l41} it is
necessary to consider the zeros of the function
  \begin{equation}\label{e47'}
  \begin{split}
g^{(2)}_{k,j,u,v}(q)&=(q^3-q)\big((10^k u ((m-1)m)^\f)_q\\
&\;\;\;-(0\overline{0^j v ( (m-1)m)^\f})_q\big)\\
&=(q+1)(q-2m)+q^{-k-1}(m-q-u+mq+u q^2)\\
  &\hspace{3cm}+q^{-j-1}(m-q-v+mq+vq^2).
  \end{split}
  \end{equation}

One can show that Lemmas  \ref{l43}--\ref{l46} also hold for $g^{(2)}_{k,j,u,v}(q)$. Here we denote by $q^{(2)}_{k,j,u,v}$ the unique root of the equation $g^{(2)}_{k,j,u,v}(q)=0$ in $(p_1,\f)$, where the parameters $(k,j,u,v)$ satisfies
$$
 \frac{u p_1+u+p_1}{p_1^{k+2}}+\frac{v p_1+v+p_1}{p_1^{j+2}}<1.
$$

\begin{lemma}
  \label{l411}
    Let $M=2m-1$. Then
     $q^{(2)}_{k,j,u,v}\in(p_1,p_2]\cap\B_2(M)$ if and only if
     $$k\in\set{2,3},\quad j=0,\quad u\in\set{0,1,\cdots,m-1},\quad v=m-1.$$
 \end{lemma}

\begin{proof}
The proof is similar to Lemma \ref{l47}. Note that $g^{(2)}_{k,j,u,v}(q)=g^{(2)}_{j,k,v,u}(q)$. Then $q^{(2)}_{k,j,u,v}=q^{(2)}_{j,k,v,u}$. By symmetry   we may assume $k\ge j$.

  First we show that $j=0$. By the monotonicity   it suffices to prove $q_{1,1,m-1,m-1}^{(2)} >p_2$. This can be verified by  using (\ref{e42})   in (\ref{e47'})   that
  \begin{equation*}
    \begin{split}
   g_{1,1,m-1,m-1}^{(2)}(p_2)={p_2}^{-2}\big((2-m){p_2}^3-3{p_2}({p_2}-m)-2({p_2}-1)\big)<0.
    \end{split}
  \end{equation*}
This implies $q_{1,1,m-1,m-1}^{(2)} >p_2$. Hence, $j=0$ as required.

Now we claim that $k\le 3$. Then   it suffices to prove $q^{(2)}_{4,0,m-1,m-1}>p_2$.
By  using (\ref{e42}) in (\ref{e47'}) it follows that
\begin{equation*}
  \begin{split}
    g^{(2)}_{4,0,m-1,m-1}(p_2)&=\frac{1-p_2}{ p_2 ^5}\;<0.
  \end{split}
\end{equation*}
Then $q^{(2)}_{4,0,m-1,m-1}>p_2$, and indeed $k\le 3$.

We  will finish the proof by considering  the following two cases.

Case I. $k=0, 1$. Then by (\ref{e47'}) we have
$$
q^{(2)}_{0,0,u,v} =\frac{2m-u-v-2}{2}+\frac{\sqrt{(2m-u- v-2)^2+4(2m-u-v)}}{2},
$$
and
$$
 q^{(2)}_{1,0,u,v} =\frac{2m-v-1}{2}+\frac{\sqrt{(2m-v-1)^2+4(m-u)}}{2}.
$$
By (\ref{e41}) and (\ref{e42})  one can verify that
\[
q^{(2)}_{0,0,u,v}\notin(p_1,p_2],\quad  q^{(2)}_{1,0,u,v}\notin(p_1,p_2]
\]
for any $u, v\in\set{0,1,\cdots, m-1}$.

Case II. $k=2, 3$. By (\ref{e41}), (\ref{e42}) and (\ref{e47'}) it follows that
\[
q_{2,0,u,v}^{(2)}\in(p_1,p_2]
\]
if and only if $v=m-1$. Moreover, one can show by (\ref{e41}), (\ref{e42}) and (\ref{e47'}) that
$q_{3,0,u,v}^{(2)}\in(p_1,p_2]$ if and only if $v=m-1$.
\end{proof}

\subsection{Solutions of  Equation (\ref{eq:43})}
Given $k, j\ge 1$ and $0\le u,v< m$, by Lemma \ref{l41} it is
necessary to consider the zeros of the function
  \begin{equation}\label{e411}
  \begin{split}
g^{(3)}_{k,j,u,v}(q)&=(q^3-q)\big( (10^k u (m(m-1))^\f)_q\\
&\;\;\;-(0\overline{0^j v ((m-1)m)^\f})_q\big)\\
&=(q+1)(q-2m)+q^{-k-1}(m-1-u+mq+u q^2)\\
  &\hspace{3cm}+q^{-j-1}(m-q-v+mq+vq^2).
  \end{split}
  \end{equation}

One can also verify that Lemmas \ref{l43}--\ref{l46} hold true for $g^{(3)}_{k,j,u,v}(q)$. Here we denote by $q^{(3)}_{k,j,u,v}$ the unique root of the equation $g^{(3)}_{k,j,u,v}(q)=0$ in $(p_1,\f)$, where $(k,j,u,v)$ satisfies
\begin{equation}
  \frac{u+1}{m p_1^{k}}+\frac{v p_1+p_1+v}{p_1^{j+2}}<1.
\end{equation}

\begin{lemma}
  \label{lem:348}
  Let $M=2m-1$. Then
   $q^{(3)}_{k,j,u,v}\in (p_1,p_2]\cap\B_2(M)$ if and only if
   $$
   k=2,\quad j=0,\quad u\in\set{0,1,\cdots,m-2},\quad v=m-1,
   $$
   or
      $$
   k=3,\quad j=0,\quad u\in\set{0,1,\cdots,m-1},\quad v=m-1,
   $$
   or $(k,j,u,v)=(4,0,m-1,m-1)$.
\end{lemma}
\begin{proof}

First we show that either $k=0$ or $j=0$. Then it suffices to prove $q^{(3)}_{1,1,m-1,m-1}>p_2$. By using (\ref{e42}) in (\ref{e411}) it follows that
\begin{equation*}
  \begin{split}
    g^{(3)}_{1,1,m-1,m-1}(p_2)&=p_2^{-2}\big((2-m)p_2^3-3(p_2^2-m p_2)-(p_2-1)\big)<0.
  \end{split}
\end{equation*}
Hence, we  have either $k=0$ or $j=0$.

Now we claim $k\ne 0$.
Suppose on the contrary that $k=0$. Then
$q^{(3)}_{0,j,u,v}\in(p_1,p_2]$ if and only if
\begin{equation}\label{e415}
g^{(3)}_{0,j,u,v}(p_1)<0,\quad  g^{(3)}_{0,j,u,v}(p_2)\ge 0.
\end{equation}
By using    (\ref{e41}) and (\ref{e42}) in (\ref{e411}) one can verify that no parameters $j,u,v$ satisfy  (\ref{e415}). Therefore,
 $ k>j=0$.

Finally, by using    (\ref{e41}) and (\ref{e42}) in (\ref{e411}) it follows that
  $q^{(3)}_{k,0,u,v}\in(p_1,p_2]$ if and only if
 $$
k=2,\quad j=0,\quad u\in\set{0,1,\cdots,m-2},\quad v=m-1,
$$
or
$$
k=3,\quad j=0,\quad u\in\set{0,1,\cdots,m-1},\quad v=m-1,
$$
or
$$(k,j,u,v)=(4,0,m-1,m-1).$$
\end{proof}

By Lemmas \ref{l47}--\ref{lem:348} we obtain the smallest base $q_2(M)$ of $\B_2(M)$ with $M=2m-1$.
\begin{theorem}
  \label{th:417}
  Let $M=2m-1$. Then
    $$(p_1,p_2]\cap\B_2(M)=\bigcup_{k=2}^3\left(\bigcup_{u=0}^{m-1}\set{q^{(2)}_{k,0,u,m-1}}\cup\bigcup_{u=0}^{m-2}\set{q^{(3)}_{k,0,u,m-1}}\right).$$

  Furthermore, the smallest base $q_2(M)$ of $\B_2(M)$ is
  $$
  q_2(M)=q_{2,0,m-1,m-1}^{(2)},
  $$
  the appropriate root of $x^4=(m-1)x^3+2 m x^2+m x+1$.
\end{theorem}
\begin{proof}
  Note by (\ref{e47'}) and (\ref{e411}) that $g^{(3)}_{k+1,0,m-1,m-1}(q)=g^{(2)}_{k,0,0,m-1}(q)$ for any $k\ge 0$. Then $$q_{k+1,0,m-1,m-1}^{(3)}=q^{(2)}_{k,0,0,m-1}.$$ Hence, by Lemmas \ref{l47}--\ref{lem:348} it follows that
  $$
  (p_1,p_2]\cap\B_2(M)=\bigcup_{k=2}^3\left(\bigcup_{u=0}^{m-1}\set{q^{(2)}_{k,0,u,m-1}}\cup\bigcup_{u=0}^{m-2}\set{q^{(3)}_{k,0,u,m-1}}\right).
  $$

 Now we consider the smallest base $q_2(M)$.  By   the monotonicity  it suffices to compare
  $$
  s:=q^{(2)}_{2,0,m-1,m-1}\quad\textrm{and}\quad t:=q^{(3)}_{2,0,m-2,m-1}.
  $$
  Note by (\ref{e47'}) and (\ref{e411}) that
 \begin{equation*}
 \begin{split}
  g^{(3)}_{2,0,m-2,m-1}(t)&=(t+1)(t-2m)+t^{-3}\big(1+m t+(m-2)t^2\big)\\
  &\hspace{3cm}+t^{-1}\big(1+(m-1)(t^2+t)\big)\\
  &<(t+1)(t-2m)+t^{-3}\big(1+(m-1)(t^2+t)\big)\\
  &\hspace{3cm}+t^{-1}\big(1+(m-1)(t^2+t)\big)\\
  &=g^{(2)}_{2,0,m-1,m-1}(t).
\end{split}
\end{equation*}
This, together with $ g^{(3)}_{2,0,m-2,m-1}(t)=0=g^{(2)}_{2,0,m-1,m-1}(s)$, implies that
$$
g^{(2)}_{2,0,m-1,m-1}(s)<g^{(2)}_{2,0,m-1,m-1}(t).
$$

 By the monotonicity of the function $g^{(2)}_{2,0,m-1,m-1}$ it  yields that
 $$ q^{(2)}_{2,0,m-1,m-1}<q^{(3)}_{2,0,m-2,m-1}.$$
  Hence,
  $q_2(M)=  q^{(2)}_{2,0,m-1,m-1} $
   satisfies the equation
  \[x^4=(m-1)x^3+2 m x^2+m x+1.\]
\end{proof}

%

\section{Proof of Theorem \ref{th:13} and final remarks}\label{sec:multiple expansions}

In this section we will prove Theorem \ref{th:13} and consider some questions on multiple expansions with multiple digits. For simplicity we write $q_2=q_2(2)$.
Recall from Section \ref{sec:unique expansion} that $I_{q_2}=[0, 2/(q_2-1)]$ and the components
\[f_k(I_{q_2})=\left[\frac{k}{q_2},\frac{k}{q_2}+\frac{2}{q_2(q_2-1)}\right],\quad k=0,1,2.\]
 Then the \emph{switch region} $ S_{q_2}$ is defined by
\begin{equation}\label{eq:51}
\begin{split}
S_{q_2}:&=\bigcup_{k=1}^2 f_{k-1}(I_{q_2})\cap f_k(I_{q_2})\\
&= \left[\frac{1}{q_2}, \frac{2}{q_2(q_2-1)}\right]\cup \left[\frac{2}{q_2}, \frac{1}{q_2}+\frac{2}{q_2(q_2-1)}\right].
\end{split}
\end{equation}

\begin{proof}[Proof of Theorem \ref{th:13}]
Clearly, $q_2$ is not an integer. Then almost every $x\in I_{q_2}$ has a continuum of different $q_2$-expansions (cf.~\cite{Sidorov_2003, Dajani_DeVries_2007}). This yields that $q_2\in\B_{2^{\aleph_0}}(2)$.

  Now  we prove $q_2\in\B_{\aleph_0}(2)$. By Theorem \ref{th:35} it gives that $q_2$ satisfies
$
q_2^2=2q_2+1.
$
This implies that
$
\al(q_2)=(20)^\f.
$
  Then
  \[
  \overline{\al(q_2)}\le \al_{i+1}(q_2)\al_{i+2}(q_2)\cdots\le \al(q_2)\quad\textrm{for all}\quad i\ge 0.
  \]
In terms of \cite[Theorem 2.6]{Komornik_Loreti_2007} it follows that $x=1\in I_{q_2}$ has countably infinitely many $q_2$-expansions, i.e.,
\[
 (20)^\f, \quad \textrm{and}\quad (20)^k 21, ~(20)^k 12^\infty  \quad\textrm{for all}\quad k\geq 0.
\]
 This establishes $q_2\in\B_{\aleph_0}(2)$.

 Finally,  we will prove $q_2\in\B_k(2)$ for all $k\ge 3$. This can be verified inductively by showing that  the number
  \[x_k=(1(00)^{k-1}1^\f)_{q_2}\]
   has exactly $k$ different $q_2$-expansions.

   If $k=1$, then by Proposition \ref{prop:23} it follows that
$x_1=(1^\f)_{q_2}$ has a unique $q_2$-expansions.

Now suppose that $x_k$ has exactly $k$ different $q_2$-expansions. Note that $q_2^2=2q_2+1$, i.e., $
(10^\infty )_{q_2}=(0210^\infty )_{q_2}.
$
This implies
\begin{equation}\label{e54}
x_{k+1}=(1(00)^{k}1^\f)_{q_2}=(021(00)^{k-1}1^\f)_{q_2}.
\end{equation}
By   Proposition \ref{prop:23} it follows that $ (00)^k 1^\f \in\U_{q_2}'$. Moreover, note that
\begin{equation*}\begin{split}
(21(00)^{k-1}1^\f)_{q_2}=\frac{2}{q_2}+\frac{1}{q^2_2}+\frac{1}{q^{2k}_2(q_2-1)}  >\frac{2}{q_2}+\frac{1}{q^2_2} =\frac{1}{q_2}+\frac{2}{q_2(q_2-1)}.
\end{split}\end{equation*}
Then by (\ref{eq:51}) this implies $(21(00)^{k-1}1^\f)_{q_2}\notin S_{q_2}$. By induction it follows that $(21(00)^{k-1}1^\f)_{q_2}$ has exactly $k$ different $q_2$-expansions. Hence, by (\ref{e54}) and Lemma \ref{lem:26} it follows that
$x_{k+1}$ has exactly  $k+1$ different $q_2$-expansions. This implies that $q_2\in \B_k(2)$ for any $k\ge 3$. Therefore, the smallest base $q_k=q_k(2)$ of $\B_k(2)$ satisfies $q_k\le q_2$.

On the other hand, for $k\ge 3$ if $x\in I_{q_k}$ has exactly $k$ different $q_k$-expansions, then by Lemma \ref{lem:26} and a linear transformation of $x$ it follows that there exists  $y\in I_{q_k}$   has exactly two different $q_k$-expansions. Hence,  $q_k\ge q_2$ for any $k\ge 3$.   This establishes the theorem.
\end{proof}

At the end of this section we consider some   questions on multiple expansions with multiple digits.
In terms of Theorems \ref{th:11} (b) and \ref{th:13}   we have the following questions:
\begin{itemize}
   \item  Does Theorem \ref{th:11} (b) holds for all $M=2m-1$?
    \item  Does Theorem \ref{th:13} holds for all $M=2m$?
\end{itemize}

The following theorem was proven by Baker and Sidorov \cite{Baker_Sidorov_2014, Baker_2015}.
\begin{theorem}
  \label{th:51}
  Let $M=1$. Then
  \begin{enumerate}[label={\rm(\alph*)}]
    \item  the smallest element of $\B_k(1)$ for $k\ge 3$ is
  \[
  q_k(1)\approx 1.75488
  \]
  the appropriate root of $x^3=2x^2-x+1$.
    \item  The second smallest element of $\B_{\aleph_0}(1)$ is
$$
q_{\aleph_0}(1)\approx 1.64541
$$
the appropriate root of the equation $x^6=x^4+x^3+2x^2+x+1$.
  \end{enumerate}
  \end{theorem}

In terms of   Theorems \ref{th:11} (a), \ref{th:13} and   \ref{th:51} we have an accurate formulae for the smallest bases $q_k(1), q_k(2)$ for   $k=2,3,\cdots$, and the second smallest base $q_{\aleph_0}(1)$. Moreover, by Theorem \ref{th:12} we have an accurate formulae for the smallest base $q_2(M)$ for all $M\ge 1$.
\begin{itemize}
  \item What is  the smallest base $q_k(M)$ of $\B_k(M)$ for $k\ge 3$ and $M\ge 3$?

  \item What is the second smallest element $q_{\aleph_0}(M)$ of $\B_{\aleph_0}(M)$ for $M\ge 2$?
\end{itemize}

\section*{Acknowledgement}

The first author  was supported by   NSFC No.  11401516 and Jiangsu Province Natural
Science Foundation for the Youth no BK20130433. The second  was
supported by  NSFC No. 11271137 and Science and Technology Commission of Shanghai Municipality (STCSM), grant No. 13dz2260400. The third author was supported by NSFC No. 11201312, 61373087; the Foundation for
Distinguished Young Teachers in Guangdong, China no Yq2013144.


\end{document}